\documentclass[reqno]{amsart}
\usepackage{float}
\usepackage{amsmath}
\usepackage{graphicx}
\usepackage{latexsym}
\usepackage{amsfonts}
\usepackage{amssymb}
\theoremstyle{plain}
\newtheorem{theorem}{Theorem}

\newtheorem{lemma}[theorem]{Lemma}
\newtheorem{proposition}[theorem]{Proposition}
\theoremstyle{definition}

\newtheorem*{remark*}{Remark}

\newcommand{\pr}{\mathbf P}
\newcommand{\e}{\mathbf E}

\begin{document}
\title[Area under excursion]{Local asymptotics for the area under the random walk excursion}

\author[Perfilev]{Elena Perfilev} 
\address{Institut f\"ur Mathematik, Universit\"at Augsburg, 86135 Augsburg, Germany}
\email{Elena.Perfilev@math.uni-augsburg.de}

\author[Wachtel]{Vitali Wachtel} 
\address{Institut f\"ur Mathematik, Universit\"at Augsburg, 86135 Augsburg, Germany}
\email{vitali.wachtel@mathematik.uni-augsburg.de}

\begin{abstract}
We study tail behaviour of the distribution of the area under the positive excursion of a
random walk which has negative drift and light-tailed increments. We determine the asymptotics
for local probabilities for the area and prove a local central limit theorem for the duration of
the excursion conditioned on the large values of its area.
\end{abstract}
\maketitle
\section{Introduction and statement of results}
Let $\{S_n;\,n\ge1\}$ be a random walk with independent, identically distributed increments
$\{X_k;\,k\ge1\}$ and let $\tau$ be the first time when $S_n$ is non-positive, i.e.,
$$
\tau:=\min\{n\ge1:S_n\le0\}.
$$
Define also the area under the trajectory $\{S_1, S_2,\ldots,S_{\tau}\}$:
$$
A_\tau:=\sum_{k=1}^{\tau-1}S_k.
$$
If the increments of the random walk have non-positive mean then the random variables $\tau$ and
$A_\tau$ are finite and we are interested in the tail behaviour of the area $A_\tau$. 

In the case of the driftless ($\mathbf{E}X_1=0$) random walk with finite variance
$\sigma^2:=\mathbf{E}X_1^2\in(0,\infty)$ one has a universal tail behaviour
\begin{equation}
\label{zero-mean}
\lim_{x\to\infty}x^{1/3}\mathbf{P}(A_\tau>x)
=2C_0\sigma^{1/3}\mathbf{E}\left(\int_0^1 e(t)dt\right)^{1/3},
\end{equation}
where $e(t)$ denotes the standard Brownian excursion and the constant $C_0$ is taken from the
relation $\mathbf{P}(\tau=n)\sim C_0 n^{-3/2}$. Proposition 1 in Vysotsky~\cite{Vysotsky10}
states that \eqref{zero-mean} holds for some particular classes of random walks. But one can
easily see that the proof from \cite{Vysotsky10} remains valid for all oscillating random walks
with finite variance. Later we shall give an alternative proof of \eqref{zero-mean}.

If the mean of $X_1$ is negative then the distribution of $A_\tau$ becomes sensitive to the tail
behaviour of the increments. Borovkov, Boxma and Palmowski \cite{BBP03} have shown that if the tail
of $X_1$ is a regularly varying function then, as $x\to\infty$,
\begin{equation}
\label{reg-var}
\mathbf{P}(A_\tau>x)
\sim\mathbf{P}\left(M_\tau>\sqrt{2|\mathbf{E}X_1|}x^{1/2}\right)
\sim\mathbf{E}\tau\mathbf{P}\left(X_1>\sqrt{2|\mathbf{E}X_1|}x^{1/2}\right),
\end{equation}
where
$$
M_\tau:=\max_{n<\tau}S_n.
$$
Behind this relation stays a simple heuristic explanation. In order to have a large area under the 
excursion the random walk has to make a large jump at the very beginning and then the random walks
behaves according to the law of large numbers. More precisely, if the jump of size $h$ appears,
after which the random walk goes linearly down with the slope $-\mu$, where $\mu:=|\mathbf{E}X_1|$,
then the duration of the excursion will be of order $h/\mu$. Consequently, the area will be of order
$h^2/2\mu$. If we want the area be of order $x$ then the jump has to be of order $\sqrt{2\mu}x^{1/2}$.
The same strategy is optimal for large values of $M_\tau$. As a result, we have both asymptotic
equivalences in \eqref{reg-var}.

This close connection between the maximum $M_\tau$ and the area $A_\tau$ is not valid for random
walks with light tails. Let $\varphi(t)$ be the moment generating function of $X_1$, that is,
$$
\varphi(t):=\mathbf{E}e^{tX_1},\quad t\ge0.
$$
We shall consider random walks satisfying the Cramer condition:
\begin{equation}
\label{cramer.cond}
\varphi(\lambda)=1\text{ for some }\lambda>0.
\end{equation}
Moreover, we shall assume that
\begin{equation}
\label{moment.cond}
\varphi'(\lambda)<\infty\quad\text{and}\quad \varphi''(\lambda)<\infty.
\end{equation}
It is well-known that if \eqref{cramer.cond} and \eqref{moment.cond} hold then the most likely
path to a large value of $M_\tau$ is piecewise linear. The random walk goes first up with the
slope $\varphi'(\lambda)/\varphi(\lambda)$. After arrival at the desired level $h$, it goes
down with the slope $-\mu$. If this path were optimal for the area then one would have
$$
\mathbf{P}(A_\tau>x)\approx\mathbf{P}\left(M_\tau>
\sqrt{\frac{2\mu\varphi'(\lambda)}{\varphi'(\lambda)+\mu\varphi(\lambda)}}x^{1/2}\right).
$$
Since $\mathbf{P}(M_\tau>y)\sim Ce^{-\lambda y}$, one arrives at the contradiction to the known
results for random walks with two-sided exponentially distributed increments, see Guillemin and
Pinchon~\cite{GP98} and Kearney~\cite{Kearney04}. 

Duffy and Meyn \cite{DM14} have shown that the optimal path to a large area is a rescaling of
the function
\begin{equation}
\label{psi.def}
\psi(u):=-\frac{1}{\lambda}\log\varphi(\lambda(1-u)),\quad u\in[0,1].
\end{equation}
They have also shown that
\begin{equation}
\label{Duffy_Meyn}
\lim_{x\to\infty} \frac{1}{x^{1/2}}\log\mathbf{P}(A_\tau>x)=-\theta,
\end{equation}
where
$$
\theta:=2\lambda\sqrt{I}\quad\text{and}\quad I:=\int_0^1\psi(u)du.
$$
Our purpose is to derive precise, without logarithmic scaling, asymptotics for 
local probabilities $\mathbf{P}(A_\tau=x)$ for integer valued random walks.
\begin{theorem}
\label{T1}
Assume that $X_1$ is integer valued and aperiodic, that is,
$$
{\rm g.c.d.}\{k-n:\mathbf{P}(X_1=k)\mathbf{P}(X_1=n)>0\}=1.
$$
Assume also that \eqref{cramer.cond} and 
\eqref{moment.cond} hold. Then there exists a positive constant $\varkappa$ such that
\begin{equation}
\label{T1.1}
\mathbf{P}(A_\tau=x)\sim \varkappa x^{-3/4}e^{-\theta \sqrt{x}},\quad x\to\infty.
\end{equation}
\end{theorem}
It is easy to see that \eqref{T1.1} implies that
\begin{equation}
\label{T1.2}
\mathbf{P}(A_\tau>x)\sim \frac{2\varkappa}{\theta}x^{-1/4}e^{-\theta \sqrt{x}}.
\end{equation}
An analogon of this relation has been obtained by Guillemin and Pinchon \cite{GP98} for an $M/M/1$
queue and by Kearney \cite{Kearney04} for a $Geo/Geo/1$ queue.

Relation \eqref{T1.2} confirms the conjecture in Kulik and Palmowski \cite{KP11} for all integer
valued random walks. Unfortunately, we do not know how to derive a version of \eqref{T1.1} for
non-lattice random walks. Moreover, we do not know how to derive \eqref{T1.2} without local asymptotics.
One can derive an upper bound for $\mathbf{P}(A_\tau>x)$ via the exponential Chebyshev inequality.
This leads to
\begin{equation}
\label{Cheb.ineq}
\mathbf{P}(A_\tau>x)\le Cx^{1/4}e^{-\theta\sqrt{x}}.
\end{equation}
For the proof of this estimate see Subsection~\ref{subsec:Cheb}. Comparing \eqref{T1.2} and
\eqref{Cheb.ineq}, we see that the Chebyshev inequality gives the right logarithmic rate of diveregence
and that the error in \eqref{Cheb.ineq} is of order $\sqrt{x}$. Such an error is quite standard for the
exponential Chebyshev inequality. In the most classical situation of sums of i.i.d. random variables
one has an error of order $\sqrt{n}$. In order to avoid this error and to obtain \eqref{T1.1} we apply
an appropriate exponential change of measure and analyse, under tarnsformed measure, the asymptotic
behavior of local probabilities for $S_n$ and $A_n:=\sum_{k=1}^nS_k$ conditioned on the event 
$\{\tau=n+1\}$. This approach allows one to obtain the following conditional limit for the duration of
the excursion.
\begin{theorem}
\label{T3}
Under the assumptions of Theorem \ref{T1}, there exists $\Delta^2>0$ such that
$$
\sup_{k}\left|x^{1/4}\mathbf{P}(\tau=k|A_\tau=x)-\frac{1}{\sqrt{2\pi\Delta^2}}
\exp\left\{-\frac{(k-I^{-1/2}x^{1/2})^2}{2\Delta^2 x^{1/2}}\right\}\right|\to0,\quad x\to\infty.
$$
\end{theorem}

\section{Non-homogeneous exponential change of measure.}
\label{sec:change_of_measure}
Our approach to the derivation of the tail asymptotics for $A_\tau$ is based on a careful analysis of
large deviation probabilities for the vector $(A_n,S_n)$ conditioned on $\{\tau=n+1\}$. For every fixed
$n$ we shall perform the following non-homogeneous change of measure. Consider a new probability measure
$\widehat{\mathbf{P}}$ such that the increments $X_1,X_2,\ldots, X_n$ are still independent and, for
every $1\le j\le n$,
\begin{equation}
\label{change.1}
\widehat{\mathbf{P}}(X_j\in dy)=
\frac{e^{u_{n,j}y}}{\varphi(u_{n,j})}\mathbf{P}(X_j\in dy),
\end{equation}
where
$$
u_{n,j}=\lambda\frac{(n-j+1)}{n}.
$$
This non-homogeneous choice of transformation parameters $u_{n,j}$ can be easily explained by the fact
that it corresponds to the exponential change of the distribution of $A_n$ with parameter $\lambda/n$.
Indeed,
$$
\mathbf{E}e^{\frac{\lambda}{n}A_n}=\mathbf{E}e^{\frac{\lambda}{n}\sum_1^n(n-j+1)X_j}
=\prod_{j=1}^n\varphi\left(\frac{n-j+1}{n}\lambda\right).
$$

We have also the following relation between probabilities $\widehat{\mathbf{P}}$ and $\mathbf{P}$:
\begin{equation}
\label{change.2}
\mathbf{P}(A_n\in dx, S_n\in dy)=e^{-\lambda x/n}\prod_{j=1}^n\varphi(u_{n,j})
\widehat{\mathbf{P}}(A_n\in dx, S_n\in dy)
\end{equation}
and
\begin{align}
\label{change.3}
\nonumber
&\pr(A_n\in dx, S_n\in dy,\tau>n)\\
&\hspace{1cm}=e^{-\lambda x/n}\prod_{j=1}^n\varphi(u_{n,j})\widehat{\pr}(A_n\in dx, S_n\in dy,\tau>n).
\end{align}
\subsection{Simple properties of the change of measure}
In this paragraph we shall collect some elementary properties of the measure defined in \eqref{change.1}.
We first note that, by the definition of $\widehat{\mathbf{P}}$,
$$
\widehat{\mathbf{E}}X_j=\frac{\varphi'(u_{n,j})}{\varphi(u_{n,j})},\quad j=1,2,\ldots,n.
$$
This implies that if $j/n\to t\in[0,1]$ then
$$
\widehat{\mathbf{E}}X_j\to\frac{\varphi'(\lambda(1-t))}{\varphi(\lambda(1-t))}\quad\text{and}\quad
\frac{1}{n}\widehat{\mathbf{E}}S_j\to\int_0^t\frac{\varphi'(\lambda(1-u))}{\varphi(\lambda(1-u))}du
=\psi(t).
$$
More precisely, there exists a constant $C$ such that, for all $j=1,2,\ldots,n$,
\begin{equation}
\label{euler.1}
\Big|\widehat{\mathbf{E}}S_j-n\psi\left(\frac{j}{n}\right)\Big|\leq C.
\end{equation}
This is statement is a standard error estimate for the Riemannian sum approximation of integrals
of a function with bounded derivative. Furthermore,
$$
\widehat{\mathbf{Var}}X_j=\frac{\varphi''(u_{n,j})}{\varphi(u_{n,j})}
-\left(\frac{\varphi'(u_{n,j})}{\varphi(u_{n,j})}\right)^2
$$
and, consequently,
$$
\frac{1}{n}\widehat{\mathbf{Var}}S_j\to
\int_0^t\left(\frac{\varphi''(\lambda(1-u))}{\varphi(\lambda(1-u))}
-\left(\frac{\varphi'(\lambda(1-u))}{\varphi(\lambda(1-u))}\right)^2\right)du. 
$$
From these asymptotics for the first two moments and from the Kolmogorov inequality we infer that
\begin{align*}
\sup_{t\in[0,1]}\left|\frac{S_{[nt]}}{n}-\psi(t)\right|\to0,\quad\text{ in }\widehat{\mathbf{P}}-\text{probability}.
\end{align*}
Fix some $\gamma\in(0,1/2)$. It is obvious that $\widehat{\mathbf{E}}X_j^3$ are uniformly bounded
for $j\in[\gamma n, (1-\gamma)n]$. This implies that the sequence $\{X_j\}_{j=1}^n$ satisfies the 
Lindeberg condition. Therefore, we have the following version of the functional central limit theorem:
the sequence of linear interplations
$$
s_n(t)=n^{-1/2}\left(S_k+X_{k}(tn-k-1)-n\psi(t)\right)\quad
\text{for }t\in\left[\frac{k}{n},\frac{k+1}{n}\right], k=0,1,\ldots,n-1
$$
converges weakly on $C[0,1]$ towards a centered gaussian process $\{\xi(t);t\in[0,1]\}$ with independent
increments and second moments
$$
\mathbf{E}(\xi(t))^2=
\int_0^t\sigma^2(u) du,
$$
where
$$
\sigma^2(u):=
\frac{\varphi''(\lambda(1-u))}{\varphi(\lambda(1-u))}
-\left(\frac{\varphi'(\lambda(1-u))}{\varphi(\lambda(1-u))}\right)^2.
$$

This functional convergence implies that
$$
\left(\frac{S_{[nt]}-n\psi(t)}{\sqrt{n}},\frac{A_{[nt]}-n^2\int_0^t\psi(s)ds}{n^{3/2}}\right)
\Rightarrow\left(\xi(t),\int_0^t\xi(s)ds\right),\quad t\in[0,1].
$$
The limiting vector has a normal distribution with zero mean. We now compute the covariance of $\xi(t)$
and $\int_0^t \xi(s)ds$. Using the independence of the increments, one can easily get
\begin{align*}
\mathbf{Cov}\left(\xi(t),\int_0^t\xi(s)ds\right)&=\int_0^t\mathbf{Cov}\left(\xi(t),\xi(s)\right)ds\\
&=\int_0^t\mathbf{Cov}\left(\xi(s)+\xi(t)-\xi(s),\xi(s)\right)ds\\
&=\int_0^t\mathbf{Cov}\left(\xi(s),\xi(s)\right)ds\\
&=\int_0^t\int_0^s\sigma^2(u)du\hspace{0,1cm}ds=\int_0^t\sigma^2(u)(t-u)du.
\end{align*}
Moreover,
\begin{align*}
\mathbf{Cov}\left(\int_0^t\xi(s)ds,\int_0^t\xi(s)ds\right)
&=\int_0^t\int_0^t \mathbf{Cov}\left(\xi(s_1),\xi(s_2)\right)ds_1 ds_2\\
&=2\int_0^tds_1\int_0^{s_1}\mathbf{Cov}\left(\xi(s_1),\xi(s_2)\right)ds_2\\
&=2\int_0^tds_1\int_0^{s_1}\left(\int_0^{s_2}\sigma^2(u)du\right)ds_2\\
&=2\int_0^t\int_0^{s_1}\sigma^2(u)(s_1-u)ds_1\hspace{0,1cm}du\\
&=\int_0^t\sigma^2(u)(t-u)^2 du.
\end{align*}
Therefore, the density of $\left(\xi(t),\int_0^t\xi(s)ds\right)$ is given by
\begin{align}
\label{t-density}
f_t(x,y):=\frac{1}{2\pi\sqrt{\det\Sigma_t}}\exp\left(-\frac{1}{2}(x,y)\Sigma_t^{-1}(x,y)^T\right).
\end{align}
with the covariance matrix
\begin{equation}
\label{t-matrix}
\Sigma_t=\begin{pmatrix}
&\int_0^t\sigma^2(u)du&\int_0^t\sigma^2(u)(t-u)du\\
&\int_0^t\sigma^2(u)(t-u)du&\int_0^t\sigma^2(u)(t-u)^2du
\end{pmatrix}.
\end{equation}

\subsection{Proof of the Chebyshev-type estimate \eqref{Cheb.ineq}}
\label{subsec:Cheb}
\begin{lemma}
\label{lem:euler}
As $n\to\infty$,
\begin{equation}
\label{euler.2}
\prod_{j=1}^n\varphi(u_{n,j})=\exp\left\{-\lambda I n\right\}(1+O(n^{-1}))
\end{equation}
\end{lemma}
\begin{proof}
It is obvious that \eqref{euler.2} is equivalent to
\begin{align}
\label{euler.22}
\sum_{j=1}^n\log\varphi(u_{n,j})=-\lambda I n+O(n^{-1}).
\end{align}
The sum on the left hand side of \eqref{euler.22} can be written as follows:
\begin{align}
\label{euler.2a}
\nonumber
\sum_{j=1}^n\log\varphi\left(\lambda\frac{n-j+1}{n}\right)
&=\sum_{j=1}^n\log\varphi\left(\lambda\left(1-\frac{j-1}{n}\right)\right)\\
\nonumber
&=-\lambda\sum_{j=1}^n\left(-\frac{1}{\lambda}\log\varphi\left(\lambda\left(1-\frac{j-1}{n}\right)\right)\right)\\
&=-\lambda\sum_{j=1}^n\psi\left(\frac{j-1}{n}\right)=-\lambda\sum_{j=0}^{n-1}\psi_n(j),
\end{align}
where $\psi_n(z):=\psi\left(\frac{z}{n}\right)$.

Applying the Euler-Mclaurin summation formula (see Gel'fond \cite{Gelfond}, p.281, formula (66)),  we obtain
\begin{align}
\label{euler.3}
\nonumber
\sum_{j=0}^{n-1}\psi_n(j)
&=\int_0^n\psi_n(t)dt+B_1\left(\psi_n(n)-\psi_n(0)\right)\\
&\hspace{1cm}-\frac{1}{2}\int_0^1\left(B_2(t)-B_2\right)\sum_{j=0}^{n-1}\psi_n''(j+1-t)dt,
\end{align}
where $B_k$  and $B_k(t)$ are Bernoulli numbers and Bernoulli polynomials respectively.

Noting that $\psi_n(n)=\psi(1)=0=\psi(0)=\psi_n(0)$, we conclude that the first correction
term in \eqref{euler.3} disappears. Furthermore, by the definition of $\psi_n$,
\begin{equation*}
\int_0^n\psi_n(t)dt=\int_0^n\psi\left(\frac{t}{n}\right)dt=n\int^1_0\psi(t)dt=nI.
\end{equation*}
Consequently, the equality \eqref{euler.3} reduces to
\begin{align}
\label{euler.4}
\sum_{j=0}^{n-1}\psi_n(j)
=nI-\frac{1}{2}\int_0^1(B_2(t)-B_2)\sum_{j=0}^{n-1}\psi_n''(j+1-t)dt.
\end{align}
Since $\varphi(z)$, $\varphi'(z)$ and $\varphi''(z)$ are bounded on the interval $[0,\lambda]$,
we get
\begin{align*}
\sup_{z\in[0,n]}\vert\psi_n''(z)\vert
=\frac{1}{n^2}\sup_{z\in[0,1]}\vert\psi''(z)\vert
=\frac{\lambda}{n^2}\sup_{z\in[0,\lambda]}\bigg\vert\frac{\varphi''(z)\varphi(z)-(\varphi'(z))^2}{\varphi^2(z)}\bigg\vert
=\frac{c}{n^2}.
\end{align*}
Therefore,
\begin{align*}
\bigg\vert\int_0^1(B_2(t)-B_2)\sum_{i=0}^{n-1}\psi_n''(j+1-t)dt\bigg\vert
\leq\frac{c}{n}\int_0^1|B_2(t)-B_2|dt=O\left(\frac{1}{n}\right).
\end{align*}
Combining this estimate with \eqref{euler.4}, we obtain
\begin{align}\label{euler.5}
\sum_{j=0}^{n-1}\psi_n(j)=nI+O\left(\frac{1}{n}\right).
\end{align}
Taking into account \eqref{euler.2a} we conclude that \eqref{euler.22} is valid.
Thus, the proof of the lemma is complete.
\end{proof}

Using \eqref{euler.2} we can derive the upper bound \eqref{Cheb.ineq} for $\mathbf{P}(A_\tau>x)$.
Obviously,
\begin{align*}
\mathbf{P}(A_{\tau}\geq x)=\sum_{n=0}^\infty\mathbf{P}(A_n\geq x,\tau=n+1).
\end{align*}
Using the exponential Chebyschev inequality and recalling that 
$$
A_n=\sum_{i=1}^n(n-j+1)X_j,
$$
we obtain
\begin{align*}
\mathbf{P}(A_n\geq x,\tau=n+1)&\leq \mathbf{P}(A_n\geq x)\\
&\leq e^{-\frac{\lambda}{n}x}\mathbf{E} e^{\frac{\lambda}{n}A_n}
=e^{-\frac{\lambda}{n}x}\prod_{i=1}^n\mathbf{E} e^{\lambda\frac{n-j+1}{n}X_j}\\
&=e^{-\frac{\lambda}{n}x}\prod_{i=1}^n\varphi\left(\lambda\frac{n-j+1}{n}\right).
\end{align*}
Applying Lemma 4, we get
\begin{align}
\label{euler.6}
\nonumber
\mathbf{P}(A_n\geq x, \tau=n+1)
&\leq\exp\left\lbrace-\lambda\frac{x}{n}-\lambda I n+O(n^{-1})\right\rbrace\\
&\le C\exp\left\lbrace-\lambda\frac{x}{n}-\lambda I n\right\rbrace.
\end{align}
Consequently,
\begin{align}\label{zerbrinsum}
\mathbf{P}(A_{\tau}\geq x)
=\sum_{n=1}^\infty\mathbf{P}(A_n\geq x,\tau=n+1)
\leq C\sum_{n=1}^\infty\exp\left\lbrace-\lambda\frac{x}{n}-\lambda I n\right\rbrace.
\end{align}
The function $t\rightarrow\lambda\frac{x}{t}+\lambda I t$ achieves its minimum at
\begin{equation}\label{mint0}
t_0=\sqrt{\frac{x}{I}}.
\end{equation}
Define: $n_{-}=\lfloor t_0\rfloor=\max\left\lbrace n\in\mathbb{N}:n\leq t_0\right\rbrace$, $n_{+}=n_{-}+1$
and split the series on the right hand side of \eqref{zerbrinsum} into two parts:
\begin{align}
\label{sum.split}
\sum_{n=1}^\infty\exp\left\lbrace-\lambda\frac{x}{n}-\lambda I n\right\rbrace
=\sum_{n=1}^{n_-}\exp\left\lbrace-\lambda\frac{x}{n}-\lambda I n\right\rbrace
+\sum_{n=n_+}^\infty\exp\left\lbrace-\lambda\frac{x}{n}-\lambda I n\right\rbrace.
\end{align}
For the second sum we have
\begin{align}
\label{secsumm}
\nonumber
&\sum_{n=n_{+}}^\infty \exp\left\lbrace-\frac{\lambda x}{n}-\lambda I n\right\rbrace\\
\nonumber
&\hspace{2cm}=\sum_{k=0}^\infty\exp\left\lbrace-\frac{\lambda x}{n_++k}-\lambda I (n_++k)\right\rbrace\\
\nonumber
&\hspace{2cm}=\exp\left\lbrace-\frac{\lambda x}{n_+}-\lambda I n_+\right\rbrace
\sum_{k=0}^\infty\exp\left\lbrace-k\lambda I+\lambda x\left(\frac{1}{n_+}-\frac{1}{n_++k}\right)\right\rbrace\\ 
\nonumber
&\hspace{2cm}=\exp\left\lbrace-\frac{\lambda x}{n_+}-\lambda I n_+\right\rbrace
\sum_{k=0}^\infty\exp\left\lbrace-k\lambda I +\frac{\lambda x k}{n_+^2(1+k/n_+)}\right\rbrace\\
\nonumber
&\hspace{2cm}\leq\exp\left\lbrace-\frac{\lambda x}{n_+}-\lambda I n_+\right\rbrace
\sum_{k=0}^\infty\exp\left\lbrace-k\lambda I+\frac{\lambda x k}{t_0^2(1+k/n_+)}\right\rbrace\\
&\hspace{2cm}=\exp\left\lbrace-\frac{\lambda x}{n_+}-\lambda I n_+\right\rbrace
\sum_{k=0}^\infty\exp\left\lbrace-k\lambda I+\frac{\lambda Ik}{(1+k/n_+)}\right\rbrace.
\end{align}
For every $k\leq n_+$ we have
\begin{equation}\label{secsumm.1}
\begin{split}
-k\lambda I+\frac{\lambda kI}{(1+k/n_+)}=-\frac{k^2\lambda I}{n_+(1+k/n_+)}\leq - \frac{k^2\lambda I}{2 n_+}.
\end{split}
\end{equation}
Therefore,
\begin{align}
\label{kfrom0tonplus}
\nonumber
\sum_{k=0}^{n_+}e^{-k\lambda I+\frac{k\lambda I}{1+k/n_+}}
&\leq 1+\sum_{k=1}^{n_+}e^{-\frac{k^2\lambda I}{2n_+}}\leq 1+\sum_{k=1}^{n_+}\int_{k-1}^k e^{-\frac{\lambda z^2 I}{2n_+}} dz\\
&\leq 1+\int_0^\infty e^{-z^2\frac{\lambda I}{2n_+}} dz=1+\sqrt{\frac{\pi}{2}}\sqrt{\frac{n_+}{\lambda I}}.
\end{align}
Furthermore, for every $k>n_+$ one has
\begin{align*}
-\frac{k^2\lambda I}{n_+(1+k/n_+)}\leq -\frac{k\lambda I}{2}.
\end{align*}
This implies that
\begin{align}
\label{kfromnplustoinfty}
\sum_{k=n_++1}^\infty e^{-k\lambda I+\frac{k\lambda I}{1+k/n_+}}
\leq \sum_{k=n_++1}^\infty e^{-\frac{k\lambda I}{2}}\leq \frac{e^{-n_+\lambda I/2}}{1-e^{-\lambda I/2}}.
\end{align}
Plugging (\ref{kfrom0tonplus}) and (\ref{kfromnplustoinfty}) into (\ref{secsumm})
and recalling the definition of $t_0$, we obtain
\begin{align}
\label{nfromnplustoinfty}
\nonumber
&\sum_{n=n_+}^\infty\exp\left\lbrace-\lambda\frac{x}{n}-n\lambda I\right\rbrace\\
\nonumber
&\hspace{1cm}\leq \exp\left\lbrace-\frac{\lambda x}{n_+}-n_+\lambda I\right\rbrace
\left( 1+\sqrt{\frac{\pi}{2}}\sqrt{\frac{n_+}{\lambda I}}+\frac{e^{-n_+\lambda I/2}}{1-e^{-\lambda I/2}}\right)\\
&\hspace{1cm}\leq C x^{1/4}\exp\left\lbrace-\frac{\lambda x}{t_0}-t_0\lambda I\right\rbrace
=C x^{1/4}\exp\left\lbrace-2\lambda\sqrt{I x}\right\rbrace.
\end{align}

We now turn to the first sum on the right hand side of \eqref{sum.split}.
Changing the summation index, we get
\begin{align}\label{nfrom1tonminus}
\nonumber
&\sum_{n=1}^{n_-}\exp\left\lbrace-\frac{\lambda}{n} x-\lambda I n\right\rbrace\\
\nonumber
&\hspace{1cm}=\sum_{k=0}^{n_--1}\exp\left\lbrace-\frac{\lambda x}{n_--k}-\lambda I (n_--k)\right\rbrace\\
&\hspace{1cm}=\exp\left\lbrace-\frac{\lambda x}{n_-}-\lambda I n_-\right\rbrace
\sum_{k=0}^{n_--1}\exp\left\lbrace k\lambda I-\frac{\lambda x k}{n_-^2(1-k/n_-)}\right\rbrace.
\end{align}
It follows from the definition of $n_-$ that $n_-^2\leq t_0^2=\frac{x}{I}$. In other words, $x\ge n_-^2 I$. 
Therefore,
\begin{align}
\label{nfrom1tonminus1}
\nonumber
&\sum_{k=0}^{n_--1}\exp\left\lbrace k\lambda I-\frac{\lambda x k}{n_-^2(1-k/n_-)}\right\rbrace\\
\nonumber
&\hspace{1cm}\leq\sum_{k=0}^{n_--1}\exp\left\lbrace-\frac{k^2\lambda I}{n_-(1-k/n_-)}\right\rbrace
\leq\sum_{k=0}^{n_--1}\exp\left\lbrace-\frac{k^2\lambda I}{n_-}\right\rbrace\\
&\hspace{1cm}\leq 1+\sum_{k=1}^{n_--1}\int_{k-1}^k \exp\left\lbrace-\frac{z^2\lambda I}{n_-}\right\rbrace dz
\leq 1+\sqrt{\frac{\pi}{2}}\sqrt{\frac{n_-}{\lambda I}}\leq C x^{1/4}.
\end{align}
Combining \eqref{nfrom1tonminus} and \eqref{nfrom1tonminus1}, we obtain
\begin{align*}
\sum_{n=1}^{n_-}\exp\left\lbrace-\frac{\lambda}{n} x-\lambda I n\right\rbrace
&\leq Cx^{1/4}\exp\left\lbrace-\frac{\lambda x}{t_0}-\lambda I t_0\right\rbrace\\
&= Cx^{1/4}\exp\left\lbrace-2\lambda\sqrt{I x}\right\rbrace.
\end{align*}
Plugging this estimate and \eqref{nfromnplustoinfty} into \eqref{sum.split}, we get
\begin{align*}
\sum_{n=1}^\infty\exp\left\lbrace-\lambda\frac{x}{n}-\lambda I n\right\rbrace
\leq C x^{1/4}\exp\left\lbrace-2\lambda\sqrt{I x}\right\rbrace.
\end{align*}
From this bound and \eqref{zerbrinsum} we obtain \eqref{Cheb.ineq}.
\section{Local limit theorems}
We start by proving a standard (unconditioned) Gnedenko local limit theorem for the two-dimenisional
vector $(S_{[nt]},A_{[nt]})$ under the measure $\widehat\pr$. The following statement is a
one-dimensional case of Theorem 4.2 in Dobrushin and Hryniv \cite{DH96} and we give its proof for
completeness reasons only.
\begin{proposition}\label{P.5}
Assume that the conditions of Theorem~\ref{T1} are valid. Then, for every $t\in(0,1]$,
\begin{align*}
\sup_{x,y}\Bigg\vert n^2\widehat\pr(S_{[nt]}=x,A_{[nt]}=y)-
f_t\left(\frac{x-n\psi(t)}{\sqrt{n}},\frac{y-n^2\int_0^t\psi(s)ds}{n^{3/2}}\right)\Bigg\vert
\longrightarrow 0,
\end{align*}
where $f_t$ is defined in \eqref{t-density} and \eqref{t-matrix}. 
\end{proposition}
\begin{proof}
Consider centered random variables
$$
X_j^0:=X_j-\widehat{\mathbf{E}}X_j
$$
and their characteristic functions
$$
\varphi_j(v):=\widehat{\mathbf{E}}e^{ivX_j^0},\quad 1\le j\le n.
$$
By the inversion formula,
\begin{align*}
&\widehat\pr (S_{[nt]}=x, A_{[nt]}=y)\\
&\hspace{1cm}=\frac{1}{(2\pi)^2}
\int_{-\pi}^\pi\int_{-\pi}^\pi e^{-iv_1x-iv_2y}\widehat{\mathbf{E}}e^{iv_1S_{[nt]}+iv_2A_{[nt]}}dv_1dv_2\\
&\hspace{1cm}=\frac{1}{(2\pi)^2}
\int_{-\pi}^\pi\int_{-\pi}^\pi e^{-iv_1n^{1/2}x_0-iv_2n^{3/2}y_0}\prod_{j=1}^n\varphi_j(v_1+(n-j+1)v_2)dv_1dv_2,
\end{align*}
where
$$
x_0:=\frac{x-\widehat{\mathbf{E}}S_{[nt]}}{n^{1/2}}
\quad\text{and}\quad
y_0:=\frac{y-\widehat{\mathbf{E}}A_{[nt]}}{n^{3/2}}.
$$
Using the change of variables $v_1\rightarrow\sqrt{n}v_1$, $v_2\rightarrow n^{3/2}v_2$, we get
\begin{align}
\label{loc.1}
\nonumber
&n^2\widehat\pr (S_{[nt]}=x, A_{[nt]}=y)\\
&\hspace{1cm}=\int_{-\pi n^{1/2}}^{\pi n^{1/2}}\int_{-\pi n^{3/2}}^{\pi n^{3/2}}e^{-iv_1x_0-iv_2y_0}
\prod_{j=1}^n\varphi_j\left(\frac{v_1}{n^{1/2}}+\frac{(n-j+1)v_2}{n^{3/2}}\right)dv_1dv_2.
\end{align}
By the same arguments,
\begin{align}
\label{loc.2}
\nonumber
&f_t\left(\frac{x-\widehat{\mathbf{E}}S_{[nt]}}{n^{1/2}},\frac{x-\widehat{\mathbf{E}}A_{[nt]}}{n^{3/2}}\right)\\
&\hspace{1cm}=\frac{1}{(2\pi)^2}\int_{-\infty}^\infty\int_{-\infty}^\infty e^{-iv_1x_0-iv_2y_0}
e^{-(v_1,v_2)\Sigma_t(v_1,v_2)^{T}}dv_1dv_2.
\end{align}
Combining \eqref{loc.1} and \eqref{loc.2}, we conclude that
\begin{align*}
&\sup_{x,y}\Bigg\vert n^2\widehat{\pr}(S_{[nt]}=x, A_{[nt]}=y)
-f_t\left(\frac{x-\e S_{[nt]}}{\sqrt{n}},\frac{y-\e A_{[nt]}}{n^{3/2}}\right)\Bigg\vert\\
&\hspace{1cm}\leq I_1+I_2+I_3+I_4,
\end{align*}
where 
\begin{align*}
I_1&=\frac{1}{(2\pi)^2}\int_{-A}^A\int_{-B}^B\bigg\vert\prod_{j=1}^n\varphi_j\left(\frac{v_1}{n^{1/2}}+\frac{(n-j+1)v_2}{n^{3/2}}\right)
-e^{(v_1,v_2)\Sigma_t(v_1,v_2)^{T}}\bigg\vert dv_1dv_2,\\
I_2&=\frac{1}{(2\pi)^2}\int_{A<\vert v_1\vert\leq\varepsilon\sqrt{n}}\int_{B<\vert v_2\vert\leq\varepsilon n^{1/2}}
\prod_{j=1}^n\left|\varphi_j\left(\frac{v_1}{n^{1/2}}+\frac{(n-j+1)v_2}{n^{3/2}}\right)\right| dv_1dv_2,\\
I_3&=\frac{1}{(2\pi)^2}\int_{\frac{|v_1|}{n^{1/2}}\in(\varepsilon,\pi]\text{ or }\frac{|v_1|}{n^{1/2}}\in(\varepsilon,\pi n]}
\prod_{j=1}^n\left|\varphi_j\left(\frac{v_1}{n^{1/2}}+\frac{(n-j+1)v_2}{n^{3/2}}\right)\right| dv_1dv_2,\\
I_4&=\frac{1}{(2\pi)^2}\int_{\vert v_1\vert>A}\int_{\vert v_2\vert>B}e^{-(v_1,v_2)\Sigma_t(v_1,v_2)^{T}}dv_1dv_2.
\end{align*}

Choosing $A$ and $B$ large enough, we can make the integral $I_4$ as small as we please. Furthermore, the weak 
convergence
$$
\left(\frac{S_{[nt]}-\widehat{\mathbf{E}}S_{[nt]}}{\sqrt{n}},\frac{A_{[nt]}-\widehat{\mathbf{E}}A_{[nt]}}{n^{3/2}}\right)
\Rightarrow\left(\xi(t),\int_0^t\xi(s)ds\right)
$$
implies that, uniformly on every compact $[-A,A]\times[-B,B]$,
$$
\bigg\vert\prod_{j=1}^n\varphi_j\left(\frac{v_1}{n^{1/2}}+\frac{(n-j+1)v_2}{n^{3/2}}\right)
-e^{(v_1,v_2)\Sigma_t(v_1,v_2)^{T}}\bigg\vert\to0.
$$
Consequently, $I_1$ converges to zero.

It is clear that the random variables $X_j^2$ are uniformly integrable with respect to
the measure $\widehat\pr$. Therefore, for every $\varepsilon$ small enough,
$$
|\varphi_j(v)|\le e^{-\sigma_j^2v^2/4},\quad |v|\le 2\varepsilon,\ 1\le j\le n.
$$
Consequently, there exist constants $c>0$ and $C$ such that
\begin{align*}
\prod_{j=1}^n\left|\varphi_j\left(\frac{v_1}{n^{1/2}}+\frac{v_2}{n^{3/2}}\right)\right|
&\leq\exp\left\lbrace-\sum_{j=1}^n\frac{\sigma_j^2}{4}\left(\frac{v_1}{\sqrt{n}}+\frac{(n-j+1)v_2}{n^{3/2}}\right)^2\right\rbrace\\
&\leq C\exp\left\lbrace-c(v_1,v_2)\Sigma_t(v_1,v_2)^{T}\right\rbrace
\end{align*}
on the set $|v_1|\le \varepsilon n^{1/2}$, $|v_2|\le\varepsilon n^{1/2}$.
Therefore, $I_2$ can be made as small as we please by choosing $A$ and $B$ large enough.\\

It remains to bound $I_3$. Since the distributions of random variables $X_j$ are
aperiodic, $\vert\widehat\e[e^{ivX_j}]\vert=1$ if and only if $v=2\pi m$. Furthermore,
recalling that the distributions of $X_j$ are obtained via the exponentail change of measure
of the same distribution and that the parameters of these changes are taken from the bounded
interval, we conclude that for every $\delta>0$ there exists $c_\delta>0$ such that
\begin{equation}
\label{loc.3}
\max_{1\le j\le n}|\varphi_j(v)|\le e^{-c_\delta}\quad\text{for all }v\text{ such that }|v-2\pi m|>\delta
\text{ for all }m\in\mathbb{Z}.
\end{equation}
For all $v_1,v_2$ from the integration region in $I_3$, we have the following property.
At least $\frac{n}{2}$ elements of the sequence
$\left\lbrace\frac{v_1}{\sqrt{n}}+\frac{(n-j+1)v_2}{n^{3/2}}\right\rbrace_{j=1}^n$
are separated from the set $\lbrace2\pi m,m\in\mathbb{Z}\rbrace$. From this fact and
\eqref{loc.3} we infer that there exists $\delta_0$ such that
$$
\prod_{j=1}^n\left|\varphi_j\left(\frac{v_1}{n^{1/2}}+\frac{v_2}{n^{3/2}}\right)\right|
\le e^{-c_{\delta_0}n/2}.
$$
Consequently, $I_3$ converges to zero as $n\to\infty$. Thus, the proof is complete.
\end{proof}

\begin{proposition}\label{T6}
Assume that the conditions of Theorem~\ref{T1} are valid. Then there exists a positive, increasing
function $q(a)$ such that, for every $t\in(0,1)$ and every $a\geq0$,
\begin{align}\label{T6.1}
\nonumber
\sup_{x,y}&\Bigg|n^2\widehat{\pr}\left(S_{[nt]}=x, A_{[nt]}=y,\min_{k\leq [nt]}S_k>-a\right)\\
&\quad -q(a)f_t\left(\frac{x-n\psi(t)}{\sqrt{n}},\frac{y-n^2\int_0^t\psi(s)ds}{n^{3/2}}\right)\Bigg|\to0.
\end{align}
\end{proposition}

\begin{proof}
Set $m=[\log^2n]$. Then, by the Markov property at time $m$,
\begin{align}
\label{T6.2}
\nonumber
&\widehat{\pr}(S_{[nt]}=x, A_{[nt]}=y,\min_{k\leq n}S_k>-a)\\
&\hspace{1cm}=\sum_{x',y'>0}\widehat{\pr}(S_m=x', A_m=y',\min_{k\leq m}S_k>-a)Q(x'y';x,y),
\end{align}
where 
\begin{align*}
&Q(x'y';x,y)\\
&\hspace{0.1cm}=\widehat{\pr}\left(S^{(m)}_{[nt]-m}=x-x', A^{(m)}_{[nt]-m}=y-y'-(n-m)x',\min_{k\leq [nt]-m}S_k^{(m)}>-x'-a\right)
\end{align*}
and
$$
S_k^{(m)}=X_{m+1}+\cdots X_{m+k}\quad\text{and}\quad A^{(m)}_k=S_1^{(m)}+S_2^{(m)}+\ldots+S_k^{(m)}.
$$

By Proposition~\ref{P.5},
\begin{align}
\label{T6.2aa}
Q(x'y';x,y) \le\widehat{\pr}\left(S^{(m)}_{[nt]-m}=x-x', A^{(m)}_{[nt]-m}=y-y'-(n-m)x'\right)
\le\frac{c_t}{n^2}.
\end{align}

It follows from the definition of $\widehat{\pr}$ that the second moments of $X_j$ are uniformly bounded.
Applying the Chebyshev inequality, we then obtain
\begin{align}
\label{T6.4a}
\widehat{\pr}(|S_m-\widehat{\mathbf{E}}S_m|\geq \log^{3/2}n)=o(1)
\end{align}
and
\begin{align}
\label{T6.4b}
\widehat{\pr}(|A_m-\widehat{\mathbf{E}}A_m|\geq \log^{5/2}n)=o(1).
\end{align}
Define
$$
D:=\left\{(x',y'):\, |x'-\widehat{\mathbf{E}}S_m|\leq \log^{3/2}n,
|y'-\widehat{\mathbf{E}}A_m|\leq \log^{5/2}n\right\}.
$$
Combining \eqref{T6.2aa}, \eqref{T6.4a} and \eqref{T6.4b}, we concude that,
uniformly in $x,y>0$,
\begin{equation}
\label{T6.2bb}
\lim_{n\to\infty} n^2\sum_{D^c}\widehat{\pr}(S_m=x', A_m=y',\min_{k\leq m}S_k>-a)Q(x'y';x,y)=0.
\end{equation}

We turn now to the asymptotic behaviour of $Q(x'y';x,y)$ for $(x',y')$ belonging to the set $D$.
Obviously,
\begin{align}
\label{T6.2a}
&Q(x'y';x,y)=\widehat{\pr}\left(S^{(m)}_{[nt]-m}=x-x', A^{(m)}_{[nt]-m}=y-y'-(n-m)x'\right)\\
\nonumber
&\hspace{0.3cm}-\widehat{\pr}\left(S^{(m)}_{[nt]-m}=x-x', A^{(m)}_{[nt]-m}=y-y'-(n-m)x',
\min_{k\leq n-m}S_k^{(m)}\leq-x'-a\right).
\end{align}
We can apply Proposition~\ref{P.5} to the first probabilty term on the right hand side of
\eqref{T6.2a}. As a result, uniformly in $x,x',y,y'>0$,
\begin{align}
\label{T6.2b}
\nonumber
&n^2\widehat{\pr}\left(S^{(m)}_{[nt]-m}=x-x', A^{(m)}_{[nt]-m}=y-y'-(n-m)x'\right)\\
&\hspace{1cm}
-f_t\left(\frac{x-x'-n\psi(t)}{\sqrt{n}},\frac{y-y'-n^2\int_0^t\psi(s)ds}{n^{3/2}}\right)\to0.
\end{align}
Furthermore, it follows easily from the definition of the measure $\widehat{\pr}$ that
$\widehat{\mathbf{E}}X_j\sim \frac{\varphi'(\lambda)}{\varphi(\lambda)}$ for each $j\le m$.
Therefore, $\widehat{\mathbf{E}}S_m\sim \frac{\varphi'(\lambda)}{\varphi(\lambda)}\log^2 n$
and $\widehat{\mathbf{E}}A_m\sim \frac{\varphi'(\lambda)}{2\varphi(\lambda)}\log^4 n$. From
these relations we infer that
\begin{align*}
&f_t\left(\frac{x-x'-n\psi(t)}{\sqrt{n}},\frac{y-y'-n^2\int_0^t\psi(s)ds}{n^{3/2}}\right)\\
&\hspace{2cm}-
f_t\left(\frac{x-n\psi(t)}{\sqrt{n}},\frac{y-n^2\int_0^t\psi(s)ds}{n^{3/2}}\right)\to0
\end{align*}
uniformly in $x,y>0$ and $(x',y')\in D$. Combining this with \eqref{T6.2b}, we conclude that
\begin{align}
\label{T6.2c}
\nonumber
&n^2\widehat{\pr}\left(S^{(m)}_{[nt]-m}=x-x', A^{(m)}_{[nt]-m}=y-y'-(n-m)x'\right)\\
&\hspace{1cm}
-f_t\left(\frac{x-n\psi(t)}{\sqrt{n}},\frac{y-n^2\int_0^t\psi(s)ds}{n^{3/2}}\right)\to0
\end{align}
uniformly in $x,y>0$ and $(x',y')\in D$.

Moreover, for every $(x',y')\in D$ and all $n$ sufficiently large we have
\begin{align*}
&\widehat{\pr}\left(S^{(m)}_{[nt]-m}=x-x', A^{(m)}_{[nt]-m}=y-y'-(n-m)x', \min_{k\leq n-m}S_k^{(m)}\leq-x'\right)\\
&\hspace{1cm}\leq\widehat{\pr}\left(\min_{k\leq [nt]-m}S_k^{(m)}\leq-x'\right)
\leq\widehat{\pr}\left(\min_{k\leq [nt]-m}S_k^{(m)}\leq-\log^{3/2}n\right).
\end{align*}
By the exponential Chebyshev inequality,
\begin{align}\label{T6.cheb.1}
\widehat{\pr}(S_k\leq-\log^{3/2}n)=\widehat{\pr}(-S_k\geq\log^{3/2}n)\leq e^{-\lambda h\log^{3/2}n}\widehat{\e}e^{-\lambda hS_k}.
\end{align}
Futhermore, it follows from the definition of $\widehat{\pr}$ that, for every $0<h<1-t$,
\begin{equation}\label{T6.cheb.2}
\begin{split}
\widehat{\e}e^{-\lambda hS_k}&=\prod_{j=1}^k\widehat{\e}e^{-\lambda h X_j}=\prod_{j=1}^k\frac{\varphi(u_{n,j}-\lambda h)}{\varphi(u_{n,j})}\\
&=\exp\left\lbrace-\lambda\sum_{j=1}^k\psi\left(\frac{j-1}{n}+h\right)+\lambda\sum_{j=1}^k\psi\left(\frac{j-1}{n}\right)\right\rbrace.
\end{split}
\end{equation}
Using here the Euler-Mclaurin summation formula \eqref{euler.3}, we infer that
\begin{align*}
\sum_{j=1}^k\psi\left(\frac{j-1}{n}\right)-\sum_{j=1}^k\psi\left(\frac{j-1}{k}+h\right)
\leq c+n\left(\int_0^{k/n}\psi(u)du-\int_h^{h+k/n}\psi(u)du\right).
\end{align*}
It is easy to see that the function $s\mapsto\int_0^s\psi(u)du-\int_h^{h+s}\psi(u)du$ achieves its maximum either at zero or at $t$.
Therefore,
\begin{align*}
\max_{s\in[0,t]}\left(\int_0^s\psi(u)du-\int_h^{s+h}\psi(u)du\right)=\left(\int_0^t\psi(u)du-\int_h^{t+h}\psi(u)du\right)^+.
\end{align*}
If $h$ is so small that $\psi(h)<\psi(t+h)$, then 
\begin{align*}
\int_0^t\psi(u)du-\int_h^{t+h}\psi(u)du<0
\end{align*}
and, consequently,
\begin{align*}
\max_{k\leq nt}\left(\sum_{j=1}^k\psi\left(\frac{j-1}{n}\right)-\sum_{j=1}^k\psi\left(\frac{j-1}{n}+h\right)\right)\leq c.
\end{align*}
Plugging this into \eqref{T6.cheb.2}, we obtain
\begin{align*}
\max_{k\leq nt}\widehat{\e}e^{-\lambda h S_k}\leq e^c.
\end{align*}
Combining this estimate and \eqref{T6.cheb.1} we finally get
\begin{align*}
\widehat{\pr}\left(\min_{k\leq nt}S_k\leq-\log^{3/2}n\right)
&\leq\sum_{j=1}^{nt}\widehat{\pr}\left(S_k<-\log^{3/2}n\right)\\
&\leq nte^ce^{-\lambda h\log^{3/2}n}=o\left(\frac{1}{n^2}\right).
\end{align*}
So we get, uniformly in $x,y>0$ and $(x',y')\in D$,
\begin{equation}
\label{T6.3}
n^2Q(x',y';x,y)-\frac{1}{n^2}f_t\left(\frac{x-n\psi(t)}{\sqrt{n}},\frac{y-n^2\int_0^t\psi(s)ds}{n^{3/2}}\right)\to0
\end{equation}
Combining \eqref{T6.3}, \eqref{T6.4a} and \eqref{T6.4b}, we conclude that
\begin{align}
\label{T6.10}
\nonumber
&n^2\sum_{D}\widehat{\pr}(S_m=x', A_m=y',\min_{k\leq m}S_k>-a)Q(x'y';x,y)\\
\nonumber
&=f_t\left(\frac{x-n\psi(t)}{\sqrt{n}},\frac{y-n^2\int_0^t\psi(s)ds}{n^{3/2}}\right)
\sum_{D}\widehat{\pr}(S_m=x', A_m=y',\min_{k\leq m}S_k>-a)+o(1)\\
&=f_t\left(\frac{x-n\psi(t)}{\sqrt{n}},\frac{y-n^2\int_0^t\psi(s)ds}{n^{3/2}}\right)
\widehat\pr\left(\min_{k\leq m}S_k>-a\right)+o(1).
\end{align}
For every fixed $m_0\ge1$ we have
\begin{align*}
\widehat{\pr}\left(\min_{k\leq m}S_k<-a\right)\le\widehat{\pr}\left(\min_{k\leq m_0}S_k<-a\right)
\end{align*}
and
\begin{align*}
\widehat{\pr}\left(\min_{k\leq m}S_k<-a\right)\ge\widehat{\pr}\left(\min_{k\leq m_0}S_k<-a\right)
-\widehat{\pr}\left(\min_{m_0< k\leq m}S_k<-a\right).
\end{align*}
For the second probability term on the right hand side we have
\begin{align*}
&\widehat{\pr}\left(\min_{m_0< k\leq m}S_k<-a\right)\\
&\hspace{1cm}\leq \widehat{\pr}\left(S_{m_0}<m_0^{2/3}\right)
+\widehat{\pr}\left(\min_{k\le m-m_0}S_k<-m_0^{2/3}\right)\\
&\hspace{1cm}\leq\widehat{\pr}\left(S_{m_0}<m_0^{2/3}\right)
+\sum_{k=1}^{m-m_0}\widehat{\pr}\left(S_k-\widehat{\e}[S_k]<-m_0^{2/3}-\widehat{\e}[S_k]\right).
\end{align*}
Using the exponential Chebyshev inequality once again, one can easily infer that there
exists $f(x)\to0$, $x\to\infty$ such that, for all $n\ge1$,
$$
\widehat{\pr}\left(\min_{m_0< k\leq m}S_k<-a\right)\le f(m_0).
$$
Consequently,
\begin{align*}
\widehat{\pr}\left(\min_{k\leq m_0}S_k>-a\right)-f(m_0)\le
\widehat{\pr}\left(\min_{k\leq m}S_k>-a\right)
\le \widehat{\pr}\left(\min_{k\leq m_0}S_k>-a\right).
\end{align*}

For every $j\le m_0$ the distribution of $X_j$ converges, as $n\to\infty$, to the distribution
of $X_1$ under $\widehat{\pr}$. (Here one has to notice that this distribution does not depend
on $n$.) Let $U_k$ denote a random walk with i.i.d. increments, which are distributed according
to the limiting distribution of $X_j$. Then
$$
\lim_{n\to\infty}\widehat{\pr}\left(\min_{k\leq m_0}S_k>-a\right)
=\mathbf{P}\left(\min_{k\leq m_0}U_k>-a\right).
$$
Letting now $m_0\to\infty$, we finally get
\begin{align*}
\lim_{n\rightarrow\infty}\widehat{\pr}\left(\min_{k\leq m}S_k>-a\right)
=\pr\left(\min_{k\geq1}U_k\geq -a\right)=:q(a).
\end{align*}
The positivity of the function $q$ follows from the fact that the increments of $U_k$ have positive mean.
Applying the previous relation to \eqref{T6.10} and taking into account \eqref{T6.2bb}, 
we obtain the desired asymptotics.
\end{proof}

In order to prove local limit theorems for $(S_n, A_n)$ conditioned on $\{tau>n,S_n=x\}$ with fixed $x$
we are going to consider the path $\{S_{[n/2]},S_{[n/2]+1},\ldots, S_n\}$ in the reversed time. More
precisely, we we shall consider random variables
$$
\widehat{X}_{k}=-X_{n-k+1}\quad\text{and}\quad
\widehat{S}_k=\widehat{X}_1+\widehat{X}_2+\ldots+\widehat{X}_k,\quad k=1,2,\ldots,n.
$$
\begin{proposition}
\label{prop.hat}
Assume that the conditions of Theorem~\ref{T1} are valid. Then there exists a positive increasing
$\widehat{q}$ such that, for every $t\in(0,1)$,
\begin{align}\label{T6.1.hat}
\nonumber
&n^2\widehat{\pr}\left(\widehat{S}_{[nt]}=x, \widehat{A}_{[nt]-1}=y,\min_{k\leq [nt]}S_k>-a\right)\\
&\hspace{2cm}
-\widehat{q}(a)\widehat{f}_t\left(\frac{x-n\psi(1-t)}{\sqrt{n}},\frac{y-n^2\int_{1-t}^1\psi(s)ds}{n^{3/2}}\right)\to0
\end{align}
uniformly in $x,y>0$. The function $\widehat{f}_t$ is the density function of the normal distribution with
zero mean and the covariance matrix
\begin{equation*}
\widehat{\Sigma}_t=\begin{pmatrix}
&\int_{1-t}^1\sigma^2(u)du&\int_{1-t}^1\sigma^2(u)(t-1+u)du\\
&\int_{1-t}^1\sigma^2(u)(t-1+u)du&\int_{1-t}^1\sigma^2(u)(t-1+u)^2du
\end{pmatrix}.
\end{equation*}
\end{proposition}
The proof of this proposition repeats that of Propositions~\ref{P.5} and \ref{T6} and we omit it.
We now state a local limit theorem for a bridge of $S_n$ conditioned to stay positive. This result
is the most important ingredient in our approach to the proof of Theorem~\ref{T1}.
\begin{proposition}
Assume that the conditions of Theorem~\ref{T1} are valid. Then, for every fixed $x$,
\begin{equation}\label{stick}
n^{2}\widehat{\pr}\left(A_n=y, S_n=x,\tau>n\right)-q(0)\widehat{q}(x)
f_1\left(0,\frac{y-n^2I}{n^{3/2}}\right)\rightarrow 0.
\end{equation}
\end{proposition}

\begin{proof}
It is immediate from the definition of $\widehat{S}_k$ that 
$S_k=S_n-\sum_{j=k+1}^nX_j=S_n+\widehat{S}_{n-k}$. Therefore, for $\ell(n)=[n t]$ with some fixed
$t\in(0,1)$ we have
\begin{align*}
\lbrace A_n=y, S_n=x\rbrace&=\left\lbrace A_{l(n)}+\sum_{k=l(n)+1}^nS_k=y, S_{l(n)}-\widehat{S}_{n-l(n)}=x\right\rbrace\\
&=\left\lbrace A_{l(n)}+(n-l(n))x+\sum_{l(n)+1}^n\widehat{S}_{n-k}=y, S_{l(n)}-\widehat{S}_{n-l(n)}=x\right\rbrace\\
&=\left\lbrace A_{l(n)}+\widehat{A}_{n-l(n)-1}=y-(n-l(n))x, S_{l(n)}-\widehat{S}_{n-l(n)}=x\right\rbrace.
\end{align*}
Consequently,
\begin{align*}
\widehat{\pr}\lbrace A_n&=y, S_n=x,\tau>n\rbrace\\
&=\widehat{\pr}\left\lbrace A_{l(n)}+\widehat{A}_{n-l(n)-1}=y-(n-l(n))x, S_{l(n)}-\widehat{S}_{n-l(n)}=x,\tau>n\right\rbrace\\
&=\sum_{x',y'}\widehat{\pr}(A_{l(n)}=y',S_{l(n)}=x',\tau>l(n))\widehat{Q}(x',y';x,y),
\end{align*}
where
\begin{align*}
&\widehat{Q}(x',y';x,y)\\
&\hspace{0.5cm}:=
\widehat{\pr}\left(\widehat{A}_{n-l(n)-1}=y-y'-(n-l(n))x, \widehat{S}_{n-l(n)}=x'-x,\min_{k\le n-\ell(n)}\widehat{S}_k>-x\right).
\end{align*}
Combining Propositions \ref{T6} and \ref{prop.hat}, we conclude that, for every fixed $x$,
\begin{align*}
n^{2}\widehat{\pr}\left(A_n=y, S_n=x,\tau>n\right)
\sim q(0)\widehat{q}(x)n^2\Sigma_n(y),
\end{align*}
where
\begin{align*}
\Sigma_n(y):=
\sum_{x',y'}f_{t}&\left(\frac{x'-n\psi(1/2)}{\sqrt{n}},\frac{y'-n^2\int_0^{1/2}\psi(s)ds}{n^{3/2}}\right)\\
&\times\widehat{f}_{1-t}\left(\frac{x'-n\psi(1/2)}{\sqrt{n}},\frac{y-y'-n^2\int_{1/2}^1\psi(s)ds}{n^{3/2}}\right).
\end{align*}
It is immediate from the continuity and boundedness of functions $f_{t}$ and $\widehat{f}_{1-t}$ that
$$
n^2\Sigma_n(y)\sim \int_{\mathbb{R}^2}f_{t}(u,v)\widehat{f}_{1-t}\left(u,\frac{y-n^2I}{n^{3/2}}-v\right)dudv,\quad n\to\infty
$$
and, consequently,
$$
n^2 \widehat{\pr}\lbrace A_n=y, S_n=x,\tau>n\rbrace
\sim q(0)\widehat{q}(x)\int_{\mathbb{R}^2}f_{t}(u,v)\widehat{f}_{1-t}\left(u,\frac{y-n^2I}{n^{3/2}}-v\right)dudv.
$$
Since the left hand side does not depend on $t$, we infer that
the integral on the right hand side does not depend on $t$ as well. Letting $t\to1$ and using continuity of $f_t$,
we infer that
$$
\int_{\mathbb{R}^2}f_{t}(u,v)\widehat{f}_{1-t}\left(u,z-v\right)dudv= f_1(0,z).
$$
This completes the proof of the proposition. 
\end{proof}

\section{Proofs of tail asymptotics}

\subsection{Proof of Theorem \ref{T1}}
Using \eqref{change.3}, we obtain
\begin{align*}
\pr(A_n=x,&\tau=n+1)\\
&=\sum_{y=1}^\infty\pr(A_n=x, S_n=y,\tau=n+1)\\
&=\sum_{y=1}^\infty\pr(A_n=x,S_n=y,\tau>n)\pr(X_{n+1}\leq -y)\\
&=e^{-\lambda x/n}\prod_{j=1}^n\varphi(u_{n,j})\sum_{y=1}^\infty\widehat{\pr}(A_n=x, S_n=y,\tau>n)\pr(X_{n+1}\leq -y).
\end{align*}
It follows from Proposition \ref{prop.hat} that, for every fixed $M$,
\begin{align}\label{A}
\nonumber
&\sum_{y=1}^M\widehat{\pr}(A_n=x, S_n=y,\tau>n)\pr(X_{n+1}\leq-y)\\
&\hspace{1cm}=\frac{q(0)}{n^2}h\left(\frac{x-n^2I}{n^{3/2}}\right)\sum_{y=1}^M\widehat{q}(y)\pr(X_1\leq-y)+o\left(\frac{1}{n^2}\right).
\end{align}
Futhermore, applying Proposition \ref{P.5}, we have
\begin{align*}
&\sum_{y=M+1}^\infty\widehat{\pr}(A_n=x, S_n=y,\tau>n)\pr(X_{n+1}\leq -y)\\
&\hspace{1cm}\leq\sum_{y=M+1}^\infty\widehat{\pr}(A_n=x,S_n=y)\pr(X_{n+1}\leq-y)
\leq\frac{c}{n^2}\sum_{M+1}^\infty\pr(X_1\leq-y).
\end{align*}
Consequently, uniformly in $n$,
\begin{align}\label{B}
\lim_{M\rightarrow\infty}n^2\sum_{y=M+1}^\infty\widehat{\pr}(A_n=x,S_n=y,\tau>n)\pr(X_{n+1}\leq-y)=0.
\end{align}
Combining \eqref{A} and \eqref{B}, we conclude that
\begin{align*}
\sum_{y=1}^\infty&\widehat{\pr}(A_n=x, S_n=y;\tau>n)\pr(X_{n+1}\leq -y)\\
&\hspace{1cm}=\frac{1}{n^2}h\left(\frac{x-n^2I}{n^{3/2}}\right)
\sum_{y=1}^\infty\widehat{q}(y)\pr(X_1\leq-y)+o\left(\frac{1}{n^2}\right).
\end{align*}
According to Lemma \ref{lem:euler},
\begin{align*}
\prod_{j=1}^n\varphi(u_{n,j})=\exp\left\lbrace-\lambda I n\right\rbrace\left(1+O(n^{-1})\right).
\end{align*}
Therefore,
\begin{align}\label{C}
\pr(A_n=x,\tau=n+1)
=\frac{Q+o(1)}{n^2}\exp\left\lbrace-\frac{\lambda x}{n}-\lambda n I\right\rbrace
h\left(\frac{x-n^2I}{n^{3/2}}\right),
\end{align}
where
$$
Q:=q(0)\sum_{y=1}^\infty\widehat{q}(y)\pr(X_1\leq-y).
$$
In particular, there exists a constant $C$, such that
\begin{align}\label{D}
\pr(A_n,\tau=n+1)\leq\frac{C}{n^2}\exp\left\lbrace-\frac{\lambda x}{n}-\lambda nI\right\rbrace.
\end{align}
Recall the definitions of $n_-$ and $n_+$. Changing the summation index and splitting the series
into two parts, we get
\begin{align}\label{E}
\nonumber
&\sum_{n=n_+}^\infty\pr(A_n=x,\tau=n+1)
=\sum_{k=0}^\infty\pr\left(A_{n_++k}=x, \tau=n_++k+1\right)\\
\nonumber
&\hspace{1cm}=\sum_{k\leq Mn_+^{1/2}}\pr(A_{n_++k}=x,\tau=n_++k+1)\\
&\hspace{3cm}+\sum_{k>Mn_+^{1/2}}\pr(A_{n_++k}=x,\tau=n_++k+1).
\end{align}
Applying \eqref{C} to the summands in the first sum, we get
\begin{align*}
&\sum_{k\leq Mn_+^{1/2}}\pr(A_{n_++k}=x,\tau=n_++k+1)\\
&=\frac{Q}{n_+^2}\sum_{k\leq Mn_+^{1/2}}\exp\left\lbrace
-\frac{\lambda x}{n_++k}-\lambda I(n_++k)\right\rbrace 
h\left(\frac{x-(n_++k)^2I}{(n_++k)^{3/2}}\right)+o\left(n_+^{-3/4}\right).
\end{align*}
Since $x-n^2_+I+k^2I=o(n^{3/2})$ uniformly in $k\leq Mn_+^{1/2}$,
\begin{align}\label{E1}
h\left(\frac{x-(n_++k)^2I}{(n_++k)^{3/2}}\right)\sim h\left(-\frac{2Ik}{n_+^{1/2}}\right).
\end{align}
Futhermore,
\begin{align*}
\frac{\lambda x}{n_++k}+\lambda I(n_++k)
&=\frac{\lambda x}{n_+}\left(1-\frac{k}{n_+}+\frac{k^2}{n_+^2}
+O\left(\frac{k^3}{n^3_+}\right)\right)+\lambda In_++\lambda Ik\\
&=\left(\frac{\lambda x}{n_+}+\lambda I n_+\right)+\lambda Ik-\frac{\lambda x}{n_+^2}k
+\frac{\lambda x k^2}{n_+^3}+O\left(\frac{\lambda x}{n_+^{5/2}}\right).
\end{align*}
Recalling now that $n_+=\sqrt{\frac{x}{I}}+\varepsilon_x$ with $\varepsilon_x\in(0,1]$, we have, 
uniformly in $k\leq Mn_+^{1/2}$,
\begin{align*}
0\leq\lambda Ik-\frac{\lambda x}{n^2_+}k
&=\left(I-\frac{x}{\frac{x}{I}\left(1+\varepsilon_x\sqrt{\frac{I}{x}}\right)^2}\right)\lambda k\\
&\leq 2\lambda k\varepsilon_x\sqrt{\frac{I}{x}}
=O\left(x^{-1/4}\right)=O\left(\frac{1}{n_+^{1/2}}\right).
\end{align*}
Consequently,
\begin{align}\label{F}
\nonumber
&\sum_{k\leq Mn_+^{1/2}}\pr(A_{n_++k}=x,\tau=n_++k+1)\\
\nonumber
&\hspace{0.5cm}=\frac{Q}{n^2_+}\exp\left\lbrace-\frac{\lambda x}{n_+}-\lambda In_+\right\rbrace
\left[\sum_{k\leq Mn_+^{1/2}}\exp\left\lbrace-\lambda I\frac{k2}{n_+}\right\rbrace 
h\left(-\frac{2IK}{n_+^{1/2}}\right)+o\left(n_+^{3/2}\right)\right]\\
\nonumber
&\hspace{0.5cm}=\frac{Q}{n^{3/2}_+}\exp\left\lbrace-\frac{\lambda x}{n_+}-\lambda In_+\right\rbrace
\left[\int_0^M e^{-\lambda Iu^2}h( -2Iu)du+o\left(n_+^{-3/2}\right)\right]\\
&\hspace{0.5cm}=\frac{\widehat{Q}}{x^{3/4}}\exp\left\lbrace-2\lambda\sqrt{Ix}\right\rbrace
\left[\int_0^Me^{-\lambda Iu^2}h(-2Iu)du+o(1)\right].
\end{align}
We split the second sum in \eqref{E} into two parts: $k\leq n_+$ and $k>n_+$.
Using \eqref{D}, we get
\begin{align*}
&\sum_{k\in(Mn_+^{1/2},n_+]}\pr(A_{n_++k}=x,\tau=n_++k+1)\\
&\hspace{1cm}\leq\frac{C}{n^2_+}\sum_{k\in(Mn_+^{1/2},n_+]}
\exp\left\lbrace-\frac{\lambda x}{n_++k}-\lambda I(n_++k)\right\rbrace.
\end{align*}
Using now \eqref{secsumm.1}, we get
\begin{align}\label{G}
\nonumber
\sum_{k\in(Mn_+^{1/2},n_+]}\pr&(A_{n_++k}=x,\tau=n_++k+1)\\
\nonumber
&\leq\frac{C}{n_+^2}\exp\left\lbrace-\frac{\lambda x}{n_+}-n_+\lambda I\right\rbrace
\sum_{k\in(Mn_+^{1/2},n_+]}e^{-\frac{\lambda I}{2}\frac{k^2}{n_+}}\\
&\leq\frac{\widehat{C}}{n_+^{3/2}}\exp\left\lbrace-\frac{\lambda x}{n_+}-n_+\lambda I\right\rbrace
\int_M^\infty e^{-\frac{\lambda Iu^2}{2}}du.
\end{align}
For $k>n_+$ we have by \eqref{euler.6} and \eqref{kfromnplustoinfty}
\begin{align}\label{H}
\nonumber
\sum_{k>n_+}\pr&(A_{n_++k}=x,\tau=n_++k+1)\\
\nonumber
&\leq\sum_{k>n_+}\exp\left\lbrace-\frac{\lambda x}{n_++k}-\lambda I(n_++k)\right\rbrace\\
&\leq C\exp\left\lbrace-\frac{\lambda x}{n_+}-\lambda In_+\right\rbrace\exp\left\lbrace-\frac{n_+\lambda I}{2}\right\rbrace.
\end{align}
Combining \eqref{F}, \eqref{G}, \eqref{H} and letting $M\rightarrow\infty$, we conclude that, for some $C_+>0$,
\begin{align*}
\sum_{n=n_+}^\infty\pr(A_n,\tau=n+1)\sim\frac{C_+}{x^{3/4}}\exp\left\lbrace-2\lambda\sqrt{Ix}\right\rbrace.
\end{align*}
Similar arguments lead to
\begin{align*}
\sum_{n=1}^{n_-}\pr(A_n=x,\tau=n+1)\sim\frac{C_-}{x^{3/4}}\exp\left\lbrace-2\lambda\sqrt{Ix}\right\rbrace.
\end{align*}
Thus the proof of Theorem \ref{T1} is complete.

\subsection{Proof of Theorem \ref{T3}}
For $k\geq 0$ we have
\begin{align*}
\pr(\tau=n_++k+1\vert A_\tau=x)=\frac{\pr(A_{n_++k}=x,\tau=n_++k+1)}{\pr(A_\tau=x)}
\end{align*}
It follows from \eqref{C} that
\begin{align*}
\pr(A_{n_++k}=x,\tau=n_++k+1)&=\frac{Q}{(n_++k)^2}\exp\left\lbrace-\frac{\lambda x}{n_++k}-\lambda I(n_++k)\right\rbrace\\
&\times\left[f_1\left(0,\frac{x-(n_++k)^2I}{(n_++k)^{3/2}}\right)+o(1)\right].
\end{align*}
It is immediate from the definition of $h$ that
$$
\varepsilon_M:=\max_{u\geq M}f_1(0,u)
\to0\quad\text{as }M\to\infty.
$$
Therefore, for all $x$ large enough and all $k\geq Mn_+^{1/2}$,
$$
\pr(\tau=n_++k+1\vert A_\tau=x)\leq C\varepsilon_M.
$$
For $k<Mn_+^{1/2}$ we have from \eqref{E1}
\begin{align*}
\pr&(A_{n_++k}=x,\tau=n_++k+1)\\
&\sim\frac{Q}{n_+^2}\exp\left\lbrace-\frac{\lambda x}{n_+}-\lambda In_+\right\rbrace 
\exp\left\lbrace-\lambda I\frac{k^2}{n_+}\right\rbrace f_1\left(0,-2I\frac{k}{n_+^{1/2}}\right).
\end{align*}
It follows now from Theorem \ref{T1} that
$$\pr(\tau=n_++k+1\vert A_\tau=x)\sim Cx^{1/4}\exp\left\lbrace -\lambda I\frac{k^2}{n_+}\right\rbrace 
f_1\left(0,-2I\frac{k}{n_+^{1/2}}\right).$$
Recalling that 
$$
f_1(0,z)=c\exp \left\{-\frac{z^2}{2\int_0^1\sigma^2(u)(1-u)^2 du}\right\}
$$
we get the desired asymptotics for $k\ge0$. The case $k<0$ can be treated in the same manner.

\subsection{Proof of \eqref{zero-mean}}
Fix some $\varepsilon>0$. Then
\begin{align}
\label{zero-mean.1}
\mathbf{P}(A_\tau>x)=\mathbf{P}(A_\tau>x,\tau\le \varepsilon x^{2/3})
+\sum_{n\ge\varepsilon x^{2/3}}\mathbf{P}(A_\tau>x,\tau=n+1).
\end{align}
It is easy to see that $\{A_\tau>x,\tau\le \varepsilon x^{2/3}\}\subset \{M_\tau>x^{1/3}/\varepsilon\}$.
Doney has shown in \cite{Doney_1983} that $y\mathbf{P}(M_\tau>y)\to c\in(0,\infty)$. Therefore, there 
exists a constant $C$ such that
\begin{align}
\label{zero-mean.2}
x^{1/3}\mathbf{P}(A_\tau>x,\tau\le \varepsilon x^{2/3})\le C\varepsilon\quad\text{for all }x>0.
\end{align}
By the functional limit theorem for random walk excursions (see Caravenna and Chaumont \cite{CC12} and
Sohier \cite{Soh10}),
$$
\mathbf{P}(A_\tau>x|\tau=n+1)=\overline{G}\left(\frac{x}{\sigma n^{3/2}}\right)+o(1),
$$
where
$$
\overline{G}(y):=\mathbf{P}\left(\int_0^1 e(t)dt>y\right).
$$
Furthermore, according to Theorem 8 in Vatutin and Wachtel \cite{VW09},
$$
\mathbf{P}(\tau=n+1)\sim\frac{C_0}{n^{3/2}}.
$$
Combining these two relations, we obtain
\begin{align*}
\mathbf{P}(A_\tau>x,\tau=n+1)=\frac{C_0}{n^{3/2}}\overline{G}\left(\frac{x}{\sigma n^{3/2}}\right)
+o(n^{-3/2}).
\end{align*}
and, consequently,
\begin{align*}
\sum_{n\ge\varepsilon x^{2/3}}\mathbf{P}(A_\tau>x,\tau=n+1)
=C_0\sum_{n\ge\varepsilon x^{2/3}}n^{-3/2}\overline{G}\left(\frac{x}{\sigma n^{3/2}}\right)
+o(x^{-1/3}).
\end{align*}
Since the sum on the right hand side can be written as a Riemannian sum for the function
$y^{-3/2}\overline{G}(y^{-3/2})$, we have
\begin{align}
\label{zero-mean.3}
\nonumber
\sum_{n\ge\varepsilon x^{2/3}}\mathbf{P}(A_\tau>x,\tau=n+1)
&=\frac{C_0\sigma^{1/3}}{x^{1/3}}\int_{\varepsilon\sigma^{2/3}}^\infty y^{-3/2}\overline{G}(y^{-3/2})dy
+o(x^{-1/3})\\
&=\frac{2C_0\sigma^{1/3}}{3x^{1/3}}\int_0^{1/(\varepsilon\sigma)} z^{-2/3}\overline{G}(z)dz
+o(x^{-1/3}).
\end{align}
Combining \eqref{zero-mean.1}--\eqref{zero-mean.3}, we obtain
$$
\liminf_{x\to\infty} x^{1/3}\mathbf{P}(A_\tau>x)\ge
\frac{2C_0\sigma^{1/3}}{3}\int_0^{1/(\varepsilon\sigma)} z^{-2/3}\overline{G}(z)dz
$$
and
$$
\limsup_{x\to\infty} x^{1/3}\mathbf{P}(A_\tau>x)\le
\frac{2C_0\sigma^{1/3}}{3}\int_0^{1/(\varepsilon\sigma)} z^{-2/3}\overline{G}(z)dz+C\varepsilon.
$$
Letting now $\varepsilon\to0$, we arrive at the relation
$$
\lim_{x\to\infty} x^{1/3}\mathbf{P}(A_\tau>x)=
\frac{2C_0\sigma}{3}\int_0^\infty z^{-2/3}\overline{G}(z)dz
=2C_0\sigma^{1/3}\mathbf{E}\left(\int_0^1 e(t)dt\right)^{1/3}.
$$


\begin{thebibliography}{99}
\bibitem{BBP03} Borovkov, A.A., Boxma, O.J. and Palmowski, Z.
\newblock On the integral of the workload process of the single server queue. 
\newblock {\em J. Appl. Probab.}, {\bf 40}:200-225, 2003.

\bibitem{CC12} Caravenna, F. and Chaumont, L.
\newblock An invariance principle for random walk bridges conditioned to stay positive.
\newblock {\em Electron. J. Probab.}, {\bf 18}, no. 60, 2013.

\bibitem{DH96} Dobrushin, R. and Hryniv, O.
\newblock Fluctuations of shapes of large areas under paths of random walks.
\newblock {\em Probab. Theory Relat. Fields}, {\bf 105}:423-458, 1996.

\bibitem{Doney_1983} Doney, R.A.
\newblock A note on conditioned random walk.
\newblock {\em  J. Appl. Probab.}, {\bf 20}, 409-412, 1983. 

\bibitem{DM14} Duffy, K.R. and Meyn, S.P.
\newblock Large deviation asymptotics for busy periods.
\newblock {\em Stochastic systems}, {\bf 4}:300-319, 2014.

\bibitem{GP98} Guillemin, F. and Pinchon, D.
\newblock On the area swept under the occupation process of an M/M/1 queue in a busy period.
\newblock {\em Queueing Syst.}, {\bf 29}383-398, 1998.

\bibitem{Kearney04} Kearney, M.J.
\newblock On a random area variable arising in discrete-time queues and compact directed percolation. 
\newblock {\em J. Phys. A, Math. Gen.}, {\bf 37}:8421-8431, 2004.

\bibitem{KP05} Kulik, R., Palmowski, Z.
\newblock Tail behaviour of the area under queue length process of a single server
queue with regularly varying service times. 
\newblock {\em Queueing Syst.}, {\bf 50}:299-323, 2005.

\bibitem{KP11}Kulik, R. and Palmowski, Z.
\newblock Tail behaviour of the area under a random process, with applications to queueing systems,
insurance and percolations.
\newblock {\em Queueing Syst.} {\bf 68}:275-284, 2011.

\bibitem{Soh10} Sohier, J.
\newblock A functional limit convergence towards brownian excursion.
\newblock Preprint, ArXiv:1012.0118, 2010.

\bibitem{VW09} Vatutin, V.A. and Wachtel, V.
\newblock Local probabilities for random walks conditioned to stay positive.
\newblock {\em Probab. Theory Relat. Fields}, {\bf 143}:177-217, 2009.

\bibitem{Vysotsky10} Vysotsky, V.
\newblock On the probability that integrated random walks stay positive.
\newblock  {\em Stochastic Process. Appl.}, {\bf 120}:1178-1193, 2010.

\bibitem{Gelfond} Gel'fond, A. O.
\newblock {\em The calculus of finite differences.}
\newblock  "Nauka'', Moscow 1967.

\bibitem{Petrov} Petrov, V.V.
\newblock {\em Sums of independent random variables.} 
\newblock  Translated from the Russian by A. A. Brown, Springer-Verlag, New York-Heidelberg, 1975. 1972.

\end{thebibliography}
\end{document}